\documentclass[a4paper,11pt]{amsart}

\usepackage[T1]{fontenc}
\usepackage[utf8]{inputenc}

\usepackage{amssymb}
\usepackage[showonlyrefs=true]{mathtools}

\makeatletter
\@namedef{subjclassname@2020}{%
  \textup{2020} Mathematics Subject Classification}
\makeatother

\allowdisplaybreaks[2]

\newtheorem{theorem}{Theorem}[section]
\newtheorem{proposition}[theorem]{Proposition}

\newtheorem{corollary}[theorem]{Corollary}
\theoremstyle{definition}

\newtheorem{definition}[theorem]{Definition}

\numberwithin{equation}{section}

\newcommand\C{\mathbb{C}}

\begin{document}

\title[Derivations and homomorphisms in commutator-simple algebras]{Derivations and homomorphisms in commutator-simple algebras}

\author{J. Alaminos}
\author{M. Bre\v sar}
\author{J. Extremera}
\author{M. L. C. Godoy}
\author{A.\,R. Villena}
\address{J. Alaminos, J. Extremera,
M. L. C. Godoy, and A.\,R. Villena, Departamento de An\' alisis
Matem\' atico, Fa\-cul\-tad de Ciencias, Universidad de Granada,
 Granada, Spain} \email{alaminos@ugr.es, 
 jlizana@ugr.es,
 mgodoy@ugr.es,
 avillena@ugr.es}
\address{M. Bre\v sar,  Faculty of Mathematics and Physics,  University of Ljubljana,
 and Faculty of Natural Sciences and Mathematics, University 
of Maribor, Slovenia} \email{matej.bresar@fmf.uni-lj.si}

\thanks{The first, third, fourth, and fifth authors were supported by the
Grant PID2021-122126NB-C31 funded by MCIN/AEI/ 10.13039/501100011033 and by ``ERDF A way of making Europe''
and grant FQM-185 funded by Junta de Andaluc\' ia.   
The second author was supported by the Slovenian Research Agency (ARRS) Grant P1-0288. 
The fourth author was also supported by grant FPU18/00419 funded by MIU}





\maketitle

\begin{abstract}
We call an algebra $A$ commutator-simple  if $[A,A]$  does not contain nonzero ideals of $A$.  After providing several examples,
we  show that in these algebras  
derivations are determined by  a condition that is applicable to the study of  local derivations. 
This enables us to prove that every continuous local derivation $D\colon L^1(G)\to L^1(G)$, 
where $G$ is a unimodular locally compact group, is a derivation.  We also give some remarks 
on homomorphism-like maps in commutator-simple algebras.
\end{abstract}

\section{Introduction} 

This paper is a continuation of \cite{Bj} in which derivations and Jordan automorphisms of semisimple finite-dimensional algebras 
$A$ were determined by certain conditions involving $[A,A]$, the linear span of all commutators in $A$. As a byproduct, new results 
on local derivations and local Jordan automorphisms were obtained.  We recall that a linear map $D$ from an algebra $A$ to itself 
is called a {\em local derivation} if for every $x\in A$ there exists a derivation $D_x \colon A\to A$ such that $D(x)=D_x(x)$. 
Other ``local maps'' 
such as local automorphisms are defined similarly. The obvious problem whether a local derivation is necessarily a derivation 
(or a local automorphism is necessarily an automorphism, etc.)  has been an active research area since the early 1990s when it 
was proposed independently by Kadison~\cite{K} and Larson and Sourour~\cite{LS}.  

We will consider similar conditions in arbitrary, not necessarily finite-dimensional algebras, as well as in Banach algebras. 
Our main motivation behind the latter has been an open question whether every continuous local derivation $D$ of $L^1(G)$, the group 
algebra of a locally compact group $G$, is a derivation. 

The following definition introduces a class of algebras to which our methods are applicable.

\begin{definition} 
An algebra $A$ is said to be \emph{commutator-simple} if $[A,A]$ does not contain nonzero ideals of $A$.
\end{definition}

Section \ref{s2} is devoted to examples of commutator-simple algebras and Banach algebras. 

In Section \ref{s3} we first show that if $A$ is a commutator-simple semiprime  algebra  $A$ (over a field of characteristic not $2$), 
then a linear map $D \colon A\to A$ satisfying $$D(x)x,D(x)x^2\in [A,A]$$ for every $x\in A$ is a  derivation. Using this, we then prove that continuous local derivations on $L^1(G)$ are derivations, provided that $G$ is a unimodular locally compact group. While this does not completely solve the aforementioned question since we still need some condition on $G$, it does generalizes and unifies several known results (see the final part of Section~\ref{s3} for details). Besides, the approach we take is entirely different from those used earlier.

The final Section~\ref{s4} discusses the condition that a linear map $T\colon A\to A$, where $A$ is a commutator-simple algebra, satisfies 
$$T(x)^3-x^3\in[A,A]$$ for every $x\in A$. Under suitable assumptions, $T$ is shown to be   a Jordan homomorphism.  Some applications 
to local automorphisms are also given.

\section{Commutator-simple algebras}\label{s2}

By an algebra we mean an associative algebra which does not necessarily possess a unity. When it does, we call it a unital algebra.
We use $F$ to denote a field, and we assume that our algebras are over $F$. When speaking of Banach algebras, we assume that $F=\C$.

The purpose of this section is to  provide examples and constructions of commutator-simple algebras.
Our first proposition is trivial, but we record it since one can construct new commutator-simple algebras from old ones (see  Propositions~\ref{dp} and~\ref{pten}).
\begin{proposition} 
Every commutative algebra is commutator-simple.
\end{proposition}

If there exists a linear functional $\tau$ on $A^2$, the linear span of all elements of the form $xy$ with $x,y\in A$, such that for all $x,y\in A$,
\begin{equation}\label{eq1}
\tau(xy)=\tau(yx),
\end{equation} 
and for all $x\in A$,
\begin{equation}\label{eq2}
  \tau(xA)=\{0\} \ \implies \ x=0,
\end{equation} 
then $A$ is obviously  a commutator-simple algebra.
The simplest example is the trace on the matrix algebra $M_n(F)$. 
Slightly more generally, we have the following.

\begin{proposition}\label{frank} 
The algebra of all finite rank  linear operators on a (finite or infinite dimensional) vector space $X$ is commutator-simple. 
\end{proposition}

\begin{proof}
This algebra has a trace $\tau$.
If $x$ is a finite rank linear operator on $X$, then 
\[
\tau\bigl(x(\xi\otimes f)\bigr)=f(x(\xi))
\]
for each  $\xi\in X$ and each linear functional $f$ on $X$. This clearly implies that~$\tau$ satisfies~\eqref{eq2}.
\end{proof}


\begin{proposition}
For any (not necessarily finite) group $G$, the group algebra $F[G]$ is commutator-simple.
\end{proposition}

\begin{proof}
For any element $x= \sum_{g\in G} \lambda_g g\in F[G]$, we define $\tau(x)$, the trace of $x$, as $\lambda_e$, the coefficient of the identity  $e\in G$. Taking another element $y=\sum_{g\in G} \mu_g g\in F[G]$,
we have 
\[
\tau(xy) = \sum_{g\in G} \lambda_g\mu_{g^{-1}} = \sum_{g\in G} \mu_g\lambda_{g^{-1}}= \tau(yx) .
\]
Thus, $\tau$ satisfies~\eqref{eq1}. Since $\tau(x g^{-1}) = \lambda_g$, it also satisfies~\eqref{eq2}.
\end{proof}

Proving that the ordinary group algebra $F[G]$ is commutator-simple is thus very easy. However, 
we are more interested in group algebras $L^1(G)$ where $G$ is a locally compact group. We will discuss this at the end of the section.

The following  proposition needs no proof.

\begin{proposition}\label{dp}
If $(A_i)_{i\in I}$ is a family of commutator-simple  algebras, then the direct product $\Pi_{i\in I} A_i$ is also commutator-simple.
\end{proposition}

We remark that a simple algebra $A$ is obviously commutator-simple if and only if $A\ne [A,A]$.

\begin{proposition} 
\label{p1} Every finite-dimensional semisimple algebra $A$ is commut\-ator-simple. 
\end{proposition}
\begin{proof} In light of Proposition~\ref{dp}, we may assume that $A$ is simple. 
The property that a unital $F$-algebra $A$ is commutator-simple does not depend on the field $F$, 
so there is no loss of generality in assuming that $A$ is a central $F$-algebra (i.e., the center of $A$ is $F$). 
Let $\overline{F}$ be the algebraic closure of $F$. It is easy to see that $A=[A,A]$ implies that $\overline{A} = \overline{F}\otimes_F A$, the scalar extension of $A$ to $\overline{F}$, also satisfies $\overline{A} =[\overline{A},\overline{A}]$. However, 
$\overline{A}$ is isomorphic to the matrix algebra $M_n(\overline{F})$  which, by Proposition \ref{frank},  does not have this property. Therefore,  
$A\ne [A,A]$, meaning that $A$ is commutator-simple.
\end{proof}

Proposition~\ref{p1} does not hold for infinite-dimensional simple algebras, in fact not even for division algebras -- see, 
for example, \cite{H}. One may therefore ask what are examples of infinite-dimensional algebras that are both simple and 
commutator-simple. 

The classical Weyl algebra $A_1$ is simple, but  not commutator-simple (as it is equal to $[A_1,A_1]$). However, many  \emph{generalized Weyl algebras} $A(F[z], a, \sigma)$ are commutator-simple. These algebras were introduced 
by V. V. Bavula \cite{Bav} to whom the authors are thankful for providing the relevant information. We recall the definition.  
Let  $a$ be an element of the polynomial algebra $F[z]$, where $F$ is a field of characteristic $0$, and let $\sigma$ be an 
automorphism of $F[z]$. The generalized Weyl algebra $A=A(F[z], a, \sigma)$ is the algebra obtained by adjoining to $F[z]$ 
the new variables $x$ and $y$ subject to the relations
\[
yx=a,\  xy= \sigma(a),
\]
and 
\[
xf = \sigma(f) x,\ fy = y\sigma(f)
\]
for all $f\in F[z]$. This algebra is simple if and only if  the difference of two roots of the polynomial $a$ is not an integer~\cite{Bav}. Further, if $\sigma$ is defined by  $\sigma(f)=f-\lambda$ with $0\ne \lambda\in F$ and the degree of $a$ is greater than $1$, then $A\ne [A,A]$ by \cite{FSS}. Therefore, the following holds.

\begin{proposition}
Let $a\in F[z]$ be a polynomial of degree at least $2$ such that the difference of two roots of  $a$ is not an integer, 
and let $\sigma$ be the automorphism of $F[z]$ defined by  $\sigma(f)=f-\lambda$ with $0\ne \lambda\in F$. Then the
generalized Weyl algebra $A(F[z], a, \sigma)$ is  commutator-simple (and simple).
\end{proposition}

Algebras from the next example are not simple.

\begin{proposition} Let $F$ be with $\operatorname{char}(F)=0$ and
 let $A=F\langle X\rangle$ be a free algebra on at least two indeterminates. 
Then $[A,A]$ does not contain nonzero subalgebras. In particular, $A$ is commutator-simple.
\end{proposition}
\begin{proof} 
Suppose there exists a nonzero $f\in F\langle X\rangle$ such that $f^j\in [A,A]$ for every $j\ge 1$. 
By~\cite[Theorem~1.7.2]{GZ}, we can pick a positive integer $n$ such that $f=f(x_1,\dots,x_m)$ is not a polynomial identity 
of $M_n(F)$. For any $A_1,\dots,A_m\in M_n(F)$, $$f(A_1,\dots,A_m)^j\in [M_n(F),M_n(F)]$$ and therefore $f(A_1,\dots,A_m)^j$  has trace $0$. 
It is well known that this implies that the matrix $f(A_1,\dots,A_m)$ is nilpotent. This means that $f^n$ is a polynomial 
identity. However, this contradicts Amitsur's theorem~\cite[Theorem~1.12.4]{GZ}.
\end{proof}

Besides the direct product, we can also use the tensor product to obtain new examples of commutator-simple algebras.

\begin{proposition}\label{pten}
Let $A$ and $B$ be  commutator-simple unital algebras.
If $A$ is simple and  central, then the algebra $A\otimes_F B$ is commutator-simple. 
\end{proposition}

\begin{proof}
Set $T=A\otimes_F B$. Suppose $[T,T]$ contains a nonzero ideal $I$. Take a nonzero
\[
w=u_1\otimes v_1 + \dots + u_m\otimes v_m\in I.
\]  
Without loss of generality, we may assume that $u_1,\dots,u_m$ are linearly independent and $v_1\ne 0$.  
Pick $a\in A\setminus [A,A]$. By the Artin-Whaples Theorem (see~\cite[Corollary~5.24]{INCA}), there exist 
$s_i, t_i\in A$ such that
\[
\sum_i s_i u_1 t_i = a\quad\mbox{and}\quad \sum_i s_i u_j t_i = 0,\,\, j=2,\dots,m.
\]
This gives
\[
a\otimes v_1 = \sum_i s_i\otimes 1 \cdot w  \cdot t_i\otimes 1 \in I,
\]
which readily implies that $a\otimes J\subseteq I$ where $J$ is the ideal of $B$ generated by~$v_1$. By assumption, there 
exists a $b\in J\setminus [B,B]$. Since $a\otimes b\in I\subseteq [T,T]$, there exist $x_k,y_k\in A$ and $z_k,w_k\in B$ such that
\[
a\otimes b = \sum_k [x_k\otimes z_k, y_k\otimes w_k] = \sum_k x_k y_k\otimes z_kw_k - y_k x_k\otimes w_kz_k,
\]
which can be written as
\[
a\otimes b +\sum_k [y_k,x_k] \otimes z_k w_k = \sum_k y_kx_k\otimes [z_k,w_k]. 
\]
Considering $\sum_k [y_k,x_k] \otimes z_k w_k$ as an element of the vector space $[A,A]\otimes B$, we can rewrite it as
$\sum_\ell c_l\otimes d_l$ for some linearly independent $c_\ell \in [A,A]$ and some $d_\ell\in B$. Since $a\notin [A,A]$, 
$a$ does not lie in the linear span of the elements $c_\ell$.  Therefore,
\[
a\otimes b +\sum_\ell c_l\otimes d_l = \sum_k y_kx_k\otimes [z_k,w_k] 
\]
implies that $b$ lies in the linear span of the elements $[z_k,w_k]$ (see~\cite[Lemma~4.9]{INCA}). However, 
this is a contradiction since  $b\notin [B,B]$.  
\end{proof}

Taking $A$ to be the matrix algebra $M_n(F)$, we obtain the following.

\begin{corollary}
If $B$ is a commutator-simple unital algebra, then so is $M_n(B)$.
\end{corollary}

We now turn our attention to Banach algebras. We start with some of those in which the argument involving the trace is applicable.

\begin{proposition}\label{cfrank}
Let $X$ be a normed space.
Then the algebra $F(X)$ of all continuous finite rank linear operators on $X$ is commutator-simple.
\end{proposition}
\begin{proof}
As in Proposition~\ref{frank} this algebra has a trace $\tau$,
and for any continuous linear operator $x$ on $X$, we have 
$\tau\bigl(x(\xi\otimes f)\bigr)=f(x(\xi))$
for each  $\xi\in X$ and each continuous linear functional $f$ on $X$.
It follows from the Hahn-Banach theorem  that $\tau$ satisfies~\eqref{eq2}.
\end{proof}

\begin{proposition}
Let $H$ be a complex Hilbert space. Then each of the following operator algebras is commutator-simple:
\begin{enumerate}
\item[(i)]
	The algebra $S^1(H)$ of all trace class operators on $H$;
\item[(ii)]
	The algebra $S^2(H)$ of all Hilbert-Schmidt operators on $H$.
\end{enumerate}	
\end{proposition}
\begin{proof}
Let $\tau$ be the natural trace on the trace class operators (see~\cite[Definition~9.1.34]{Pal}).
By~\cite[Theorem~9.1.35]{Pal}, $F(H)\subset S^1(H)\subset S^2(H)$ and
$\tau(xy)=\tau(yx)$ for all $x$, $y\in S^2(H)$. This shows that $\tau$ satisfies~\eqref{eq1} in both cases (i) and (ii). Further,
\[
\tau\bigl(x(\xi\otimes\eta^*)\bigr)=
\tau\bigl(\xi(x)\otimes\eta^*\bigr)=
\langle\xi(x)\vert\eta\rangle
\]
for each continuous linear operator $x$ and all $\xi,\eta\in H$, which
gives~\eqref{eq2}.
\end{proof}

Let $G$ be a locally compact group.
We write $L^1(G)$ for the usual Banach $L^1$-space with respect to the left Haar measure on $G$. It becomes a Banach algebra with respect to the convolution product defined by
\[
(f\ast g)(t)=\int_Gf(s)g(s^{-1}t)\,ds.
\]
If $G$ is a discrete group, then the Haar measure is the counting measure, and
the corresponding group algebra is usually written as $\ell^1(G)$. This algebra
coincides with the purely algebraic group algebra $\mathbb{C}[G]$ in the case where $G$ is a finite group.
The group $G$ is called \emph{unimodular} if the left Haar measure is also  a right Haar measure.
This is a very important and wide class of groups, which includes 
Abelian groups, compact groups, discrete groups, and many others 
(we refer the reader to~\cite[Chapter12]{Pal} for examples and a thorough
discussion of this class of groups).
We confine our  attention to unimodular locally compact groups.
In this case we have
\[
\int_G f(t^{-1})\, dt=\int_G f(t)\,dt
\]
for each $f\in L^1(G)$, and
the group algebra $L^1(G)$ turns into a Banach $\star$-algebra by defining
\[
f^\star(t)=\overline{f(t^{-1})} 
\]
for all $ f\in L^1(G)$ and $t\in G$.

\begin{proposition}\label{g1}
Let $G$ be a unimodular locally compact group. 
Then each of the following subalgebras of the group algebra $L^1(G)$ is commutator-simple:
\begin{enumerate}
\item[(i)] 
The algebra $C_{00}(G)$ of continuous functions on $G$ of compact support.
\item[(ii)]
The algebra $L^1(G)\cap L^2(G)$.
\end{enumerate}
\end{proposition}

\begin{proof}
Let $e$ be the identity of the group $G$.

(i)
It is easy to check that $C_{00}(G)$ is a $\star$-subalgebra of $L^1(G)$. 
We define a linear functional $\tau\colon C_{00}(G)\to\mathbb{C}$ by
\[
\tau(f)=f(e).
\]
Using the unimodularity of $G$, for each $f,g\in C_{00}(G)$, we have
\[
\begin{split}
\tau(f\ast g) & = (f\ast g)(e) =
\int_{G}f(t)g(t^{-1}e)\, dt =
\int_{G}f(t)g(t^{-1})\,dt \\
& = \int_Gf(t^{-1})g(t)\, dt=
\int_G g(t)f(t^{-1}e)\, dt=
(g\ast f)(e)=\tau(g\ast f),
\end{split}
\]
which shows that $\tau$ is a trace on $C_{00}(G)$. 
Further, suppose that $f\in C_{00}(G)$ is such that $\tau(f\ast g)=0$ for each $g\in C_{00}(G)$. We have
\[
0=\tau(f\ast f^\star)=
(f\ast f^\star)(e)=\int_G f(t)f^\star(t^{-1}e)\,dt=
\int_G\vert f(t)\vert^2\,dt,
\]
whence $f=0$.
	
(ii)
Write $A=L^1(G)\cap L^2(G)$.
Since $G$ is unimodular, \cite[Proposition~2.40]{F} shows that $A$ is a subalgebra of $L^1(G)$, and it is easily seen that it is $\star$-invariant.
Further,
if $f,g\in A$, then \cite[Proposition~2.41]{F} shows that $f\ast g$ is defined everywhere on $G$ 
and it is a continuous function vanishing at infinity. This implies that we can define a linear functional
$\tau\colon A^2\to\mathbb{C}$ by $\tau(f)=f(e)$ for each $f\in A^2$.
We can check that both~\eqref{eq1} and~\eqref{eq2} hold for $\tau$ as in~(i).
\end{proof}

\begin{proposition}\label{g2}
Let $G$ be a locally compact group.
Then the group algebra $L^1(G)$ is commutator-simple in each of the following cases:
\begin{enumerate}
\item[(i)]
The group $G$ has \emph{small invariant neighborhoods}, i.e.,
every neighborhood of the identity contains a compact neighborhood of the identity which is invariant under all inner automorphisms;

\item[(ii)]
The group $G$ is \emph{maximally almost periodic}, i.e.,
the are enough continuous finite-dimensional unitary representations to separate the points of $G$.

\end{enumerate}
\end{proposition}
\begin{proof}
(i)
Let $(e_\lambda)_{\lambda\in\Lambda}$ be a bounded approximate identity for $L^1(G)$ consisting of 
central  elements of $L^1(G)$ and such that $e_\lambda\in L^1(G)\cap L^2(G)$ for each $\lambda\in\Lambda$ 
(see~\cite[Page~530]{Pal1}).

Suppose that $I$ is an ideal of $L^1(G)$ with $I\subseteq[L^1(G),L^1(G)]$.
Then, for each $\lambda\in\Lambda$,
\[
L^1(G)\ast e_\lambda\subseteq L^1(G)\cap L^2(G)
\]
and so
\[
I\ast e_\lambda\ast e_\lambda\subseteq [L^1(G)\ast e_\lambda,L^1(G)\ast e_\lambda]
\subseteq [L^1(G)\cap L^2(G),L^1(G)\cap L^2(G)].
\]
Further, $I\ast e_\lambda\ast e_\lambda$ is an ideal of $L^1(G)\cap L^2(G)$. From Proposition~\ref{g1} 
we see that $I\ast e_\lambda\ast e_\lambda=\{0\}$. 
Finally, since $(e_\lambda)_{\lambda\in\Lambda}$ is a bounded approximate identity for $L^1(G)$, for each $f\in I$ we have
\[
f=\lim_{\lambda\in\Lambda} \underbrace{f\ast e_\lambda\ast e_\lambda}_{=0}=0.
\]

(ii)
Suppose that $I$ is an ideal of $L^1(G)$ such that $I\subseteq [L^1(G),L^1(G)]$,
and take $f\in I$.

Let $\pi$ be a continuous, unitary representation of $G$ on a finite-dimensional Hilbert space $H$. 
Then $\pi$ gives rise to a representation of $L^1(G)$ on the algebra $B(H)$ of all continuous linear operators on $H$, still denoted by $\pi$, defined by
\[
\pi(g)=\int_G f(t)\pi(t)\, dt 
\]
for all $g\in L^1(G)$.
Since $H$ is finite-dimensional we can define a linear functional $\tau_\pi\colon L^1(G)\to\mathbb{C}$ by
\[
\tau_\pi(g)=\operatorname{trace}(\pi(g))
\]
for all $ g\in L^1(G)$.
It is immediate to check that $\tau_\pi$ is a trace.
Consequently, $\tau_\pi(I)=\{0\}$ and hence
\[
0=\tau_\pi(f\ast f^\star)=\operatorname{trace}(\pi(f)\pi(f)^*),
\]
which implies that $\pi(f)=0$.

We thus get $\pi(f)=0$ for each finite-dimensional, continuous, unitary representation $\pi$ of $G$,
and since $G$ is an MAP-group, it may be concluded that $f=0$.
\end{proof}

For a comprehensive treatment of the classes of groups, [SIN] and [MAP], appearing in Proposition~\ref{g2}, 
we refer the reader to \cite[Chapter~12]{Pal}. It should be pointed out that all of these groups are unimodular.
Further, 
each discrete group $G$ has small invariant neighborhoods and therefore the algebra $\ell^1(G)$ is commutator-simple.

\section{Derivations}\label{s3}

Let $A$ be an algebra.
A linear map $D\colon A\to A$ is called a \emph{Jordan derivation} if it satisfies
\[
D(x^2)=D(x)x + xD(x)
\]
for all $x\in A$. If $A$ is a \emph{semiprime} algebra, i.e.,
$A$ has no nonzero nilpotent ideals, and the characteristic 
of the underlying field is not $2$,
every Jordan derivation is automatically a \emph{derivation}~\cite{C}, meaning that it satisfies 
\[
D(xy)=D(x)y + xD(y)
\]
for all $x,y\in A$. 

We remark that there are many algebras in which every derivation $D$ is of the form
\[
D(x)=[x,m]
\]
where $m$ is a fixed element from $A$, or, sometimes, from a larger algebra $M$ that contains $A$ as an ideal.
These are the derivations that are of particular interest to us. Observe that they satisfy
\begin{equation}\label{xnan}
D(x)x^n = [x,m]x ^n = [x,mx^n]\in [A,A]
\end{equation}
for all $x\in A$ and all positive integers $n$. We will actually need this only for $n=1$ and $n=2$.

We now state our basic result on derivations in commutator-simple algebras.  Its proof is short and uses an idea from~\cite{Bj}.

\begin{theorem}\label{l3}
Let $A$ be a commutator-simple semiprime algebra over a field $F$ with $\operatorname{char}(F)\ne 2$. If a linear map $D\colon A\to A$ satisfies 
\[
D(x)x, D(x)x^2\in [A,A]
\] 
for every $x\in A$, then $D$ is a  derivation.
\end{theorem}
\begin{proof}
For any $x,y\in A$, we write 
 $x\equiv y$  if $ x-y\in [A,A].$ Our assumption can thus be read as $D(x)x\equiv 0$ and  $D(x)x^2\equiv 0$ for every $x\in A$. Writing $x+y$ for $x$ in the first relation we obtain
\begin{equation} 
D(x)y \equiv - D(y)x\label{31}
\end{equation}
for all $x,y\in A$, 
and writing $x\pm y$ for $x$ in the second relation, we obtain
\begin{equation} 
D(x)(xy+yx) + D(y)x^2  \equiv 0 \label{32}
\end{equation}
for all $x,y\in A$ (here we used the assumption that  $\operatorname{char}(F)\ne 2$).
In light of \eqref{31}, \eqref{32} can be written as 
\begin{equation} 
D(x)(xy+yx) -  D(x^2 )y \equiv 0. \label{33}
\end{equation}
Since 
\[
D(x)yx - xD(x)y = [D(x)y,x] \equiv 0,
\] 
\eqref{33} further gives
\begin{equation}
\big (D(x)x + xD(x) - D(x^2)\big)y \equiv 0. \label{34}
 \end{equation}
 
Suppose $D$ is not a  derivation.
By the aforementioned result, $D$  is neither a Jordan derivation, so
\[
a= D(x)x + xD(x) - D(x^2)\ne 0
\] 
for some $x\in A$.
By \eqref{34}, $ay\equiv 0$ for all $y\in A$. Since $zay = [z,ay] + a(yz)$ it follows that $zay\equiv 0$ for all $y,z\in A$. 
This means that $[A,A]$ contains the ideal $AaA$, which is nonzero since $A$ is semiprime. This contradicts our assumption that  
$A$ is commutator-simple.
\end{proof}

Theorem \ref{l3} is of course applicable to most algebras mentioned in Section~\ref{s2}, but  we will not list 
them here and rather focus on local derivations of group algebras $L^1(G)$. For this purpose, we need the following 
technical corollary to Theorem~\ref{l3}.

\begin{corollary}\label{p1209alternative}
	Let $A$  be a semiprime algebra over a field $F$ with $\operatorname{char}(F)\ne 2$. Suppose $A$ is an ideal of an algebra $M$ such that every derivation from $A$ to $A$ is of the form $x\mapsto [x,m]$ for some $m\in M$. Then every local derivation  $D\colon A\to A$ is a derivation on each commutator-simple ideal $I$ of $M$ contained in $A$.
\end{corollary}
\begin{proof}
	Take $x\in I$.  Since $x\in A$, by our assumption there exists an $m_x\in M$ such that 
	 $D(x)=[x,m_x]\in I$. Therefore, $D$ maps $I$ into itself, and just as in~\eqref{xnan} we see that 
	 \[
D(x)x^n = [x,m_x]x ^n = [x,m_xx^n]\in [I,I]
\]
for every positive integer $n$. As an ideal of a semiprime algebra, $I$ itself is semiprime. Applying Theorem~\ref{l3} to the algebra $I$ it follows that  $D$ is a derivation on $I$.
\end{proof}

\begin{theorem}\label{locder}
Let $G$ be a unimodular locally compact group. 
Then every continuous local derivation $D\colon L^1(G)\to L^1(G)$ is a derivation. 
\end{theorem}
\begin{proof}
We will check that Corollary~\ref{p1209alternative} applies with
$A=L^1(G)$ and $M=M(G)$, the Banach algebra of complex Radon measures on $G$. Of course, $A$ is a closed ideal of $M$.

A remarkable property of the group algebra $L^1(G)$ is that each derivation from $L^1(G)$ to itself is of the form $f\mapsto f\ast\mu-\mu\ast f$ for some $\mu\in M(G)$ (see~\cite{BGM, losert}).
We take  $I=L^1(G)\cap L^2(G)$, which is an ideal of $A$ and,  
by Proposition~\ref{g1}, $I$ is a commutator-simple algebra. 
From Corollary~\ref{p1209alternative} we see that $D$ is a derivation on $I$.
Since $I$  is dense in $L^1(G)$ and $D$ is continuous on $L^1(G)$, we conclude that $D$ is a derivation on $L^1(G)$.
\end{proof}

In \cite{Sam1, Sam2}, the author shows that every continuous (approximately) local derivation from $L^1(G)$ to
any Banach $L^1(G)$-bimodule $X$ is a derivation for a large class of groups including 
PG-groups, IN-groups, MAP-groups, and totally disconnected groups. In order to relate this result to Theorem~\ref{locder}, we  mention that the class of unimodular groups strictly contains the classes [PG], [IN], and [MAP]
(we refer the reader to~\cite[Diagram~1, page~1486]{Pal} for this fact),
and, on the other hand, unlike \cite{Sam1, Sam2}, we are confined to maps from $L^1(G)$ to itself.
To the best of our knowledge, Theorem~\ref{locder} gives a new information.

\section{Jordan homomorphisms}\label{s4}

A linear map $T$ from an algebra $A$ to itself is called a \emph{Jordan homomorphism} if 
\[
T(x^2) = T(x)^2
\]
for every $x\in A$. Basic examples are homomorphisms and antihomomorphisms and, under suitable conditions, these are also the only examples. However, we shall not go into this here.

The idea of the proof of the following theorem is also taken from~\cite{Bj}.  However, there are important differences in details.

\begin{theorem}\label{p}
Let $A$ be a commutator-simple unital  algebra over a field $F$ with $\operatorname{char}(F)\ne 2,3$. If a surjective linear map $T\colon A\to A$ satisfies 
$T(1)=1$ and   $T(x)^3- x^3\in [A,A]$ for every $x\in A$, then $T$ is a Jordan homomorphism.
\end{theorem}
\begin{proof} As above, we write $x\equiv y$ for $x-y\in [A,A]$. 
Since char$(F)\ne 2$, replacing $x$ by $x\pm y$ in $T(x)^3\equiv x^3$ gives
\[
T(x)^2 T(y) + T(x)T(y)T(x) + T(y)T(x)^2 \equiv  x^2 y + xyx + yx^2.
\]
Observing that 
\begin{gather*}
T(x)^2 T(y) \equiv T(x)T(y)T(x) \equiv T(y)T(x)^2, \\
 x^2 y \equiv xyx \equiv yx^2,
\end{gather*} 
and using $\operatorname{char}(F)\ne 3$  it  follows that
\begin{equation*} 
T(x)^2T(y)\equiv x^2 y
\end{equation*}
for all $x,y\in A$. Hence,
\[
T(x)^2T(y^2) \equiv x^2 y^2 \equiv y^2x^2 \equiv T(y)^2 T(x^2) \equiv T(x^2) T(y)^2.
\]
Replacing $y$ by $y+ 1$ in  $T(x)^2T(y^2)=T(x)^2T(y)^2$ and using $T(1)=1$ we arrive at
\[
T(x)^2T(y) \equiv T(x^2) T(y)
\]
for all $x,y\in A$. Since $T$ is surjective, this means that $$\big(T(x^2)- T(x)^2\big)A\subseteq [A,A],$$ which,
by  $zay = [z,ay] + a(yz)$, implies that 
\[
A\big(T(x^2)- T(x)^2\big)A\subseteq [A,A].
\] 
Since $A$ is commutator-simple, it follows that $T(x^2)- T(x)^2=0$.
\end{proof}

Let $A$ be a unital algebra. By a \emph{local inner automorphism} of $A$ we, of course, mean a linear map $T\colon A\to A$ 
such that, for every $x\in A$, there exists 
an inner automorphism $T_x$ of $A$  such that $T(x)=T_x(x)$;
that is, $T(x)$ is always of the form  $u_xxu_x^{-1}$ for some invertible $u_x\in A$.
 Observe that such a map is automatically injective and sends $1$ to $1$.

\begin{corollary}\label{c}
If $A$ is a commutator-simple unital  algebra over a field $F$ with $\operatorname{char}(F)\ne 2,3$, 
then every surjective local inner automorphism of $A$  is a  Jordan automorphism.
\end{corollary}
\begin{proof} 
Observe that $$uxu^{-1} - x = [u,xu^{-1}]\equiv 0$$ for every invertible $u\in A$ and every $x\in A$. Hence 
\[
T(x)^3 = (u_xxu_x^{-1})^3 = u_xx^3u_x^{-1} \equiv x^3
\] 
for every $x\in A$, and so Theorem~\ref{p} applies.
\end{proof}

Every matrix $x$ is similar to its transpose $x^t$, so $x\mapsto x^t$ is an example of a surjective local 
inner automorphism of the matrix algebra $M_n(F)$ that is a Jordan automorphism but not an automorphism. 
We also remark that since all automorphisms of this algebra are automatically  inner, the notion of a local 
inner automorphism here coincide with the notion of a local automorphism.

We also need a technical refinement of this corollary that concerns algebras that are not necessarily  unital.

\begin{corollary}\label{o0}
Let $A$ be a commutator-simple algebra over a field $F$ with $\operatorname{char}(F)\ne 2,3$.
Assume further that $y\in Ay$ for each $y\in A$ and that $A$ is an ideal of a unital algebra $M$. 
If $T\colon A\to A$ is a surjective linear map such that
for each $y\in A$ there exists an invertible element $u_y\in M$ satisfying $T(y)=u_yyu_y^{-1}$, 
then $T$ is a  Jordan automorphism.
\end{corollary}
\begin{proof}
Write $A_1=F1+A$.
We claim that $A_1$ is a commutator-simple algebra. Indeed,  let $I$ be an ideal of $A_1$ such that
$I\subseteq [A_1,A_1]$. Since $[A_1,A_1]=[A,A]\subseteq A$  we see that  $I$ is an ideal of $A$ contained in $[A,A]$, and so $I=\{0\}$.

Define $T_1\colon A_1\to A_1$ by 
\[
T_1(\alpha 1+y)=\alpha 1+T(y)
\] 
for all $\alpha\in F$ and $y\in A$. Observe that
$T_1$ is a well-defined linear map. Moreover, it is surjective and $T_1(1)=1$.
We claim that $T_1(x)^3\equiv x^3$ for each $x\in A_1$.
Set $x=\alpha 1+y\in A_1$.
Take $e\in A$ such that $y=ey$ and write  $u = u_y\in M$.
Then
\begin{align*}
T_1(x)^3-x^3 & = (\alpha 1 + uyu^{-1})^3 - (\alpha 1 + y) ^3
\\
& =3\alpha^2(uyu^{-1}-y)+3\alpha(uy^2u^{-1}-y^2)+(uy^3u^{-1}-y^3)
\\ & =
3\alpha^2[ue,yu^{-1}]+3\alpha^2[y,e]+3\alpha[uy,yu^{-1}]+[uy,y^2u^{-1}]\in [A,A],
\end{align*}
as claimed.
By Theorem~\ref{p} it follows that $T_1$ is a Jordan automorphism and hence $T$ is a Jordan automorphism.
\end{proof}

Let $A$ be an algebra of linear operators on a vector space that contains all finite rank operator. As a local automorphism of 
$A$ preserves the rank of any finite rank operator, it seems likely that its form can be determined by using well known results 
on rank preservers. What we will establish in the next lines could therefore be obtained by other means, or may even be already 
known. However, our method of proof is certainly new.

\begin{corollary}\label{o1}
Let $A$ be the algebra of all finite rank linear operators on a vector space $X$ over a field $F$ with $\operatorname{char}(F)\ne 2,3$.
Then every surjective local automorphism of $A$ is a  Jordan automorphism.
\end{corollary}
\begin{proof} 
Let $M$ be the algebra of all linear operators on $X$,
and we will check that all requirements in Corollary~\ref{o0} are satisfied.
By Proposition~\ref{frank}, $A$ is commutator-simple.
For each $x\in A$, we can choose an idempotent operator $P\colon X\to X$ mapping onto the range
of $x$, which implies that $P\in A$ and $x=Px$, so that $x\in Ax$. Further, $M$ is a unital algebra 
and $A$ is an ideal of $M$.
For each $x\in A$ there exists an automorphism $\Phi_x\colon A\to A$ such that $T(x)=\Phi_x(x)$,
and~\cite[Section~IV.11]{Jac} shows that there exists an invertible linear operator $u_x\colon X\to X$
such that $\Phi_x(y)=u_xyu_x^{-1}$ for each $y\in A$. Proposition~\ref{o0} gives the desired conclusion.
\end{proof}

\begin{corollary}\label{o2}
Let $X$ be a normed space.
Then every surjective local automorphism of $F(X)$ is a Jordan automorphism.
\end{corollary}
\begin{proof}
Set $A=F(X)$, and let $M$ be the algebra of all continuous linear operators on $X$.
Then Proposition~\ref{cfrank} shows that $A$ is a commutator-simple algebra.
If $x\in A$, then the range $Y$ of $x$ is finite-dimensional and therefore there exists
a continuous linear projection $P$ of $X$ onto $Y$. Consequently, $P\in A$ and $x=Px$, which
gives $x\in Ax$.
Moreover, $M$ is a unital algebra and $A$ is an ideal of $M$.
Take $x\in A$.
Then there exists an automorphism $\Phi_x\colon F(X)\to F(X)$ such that
$T(x)=\Phi_x(x)$. By~\cite[Theorem~3.1]{Cher}, 
there exists a continuous invertible linear operator $u_x\colon X\to X$ such that
$\Phi_x(y)=u_x yu_x^{-1}$ for each $y\in F(X)$.
From Proposition~\ref{o0} we obtain that $T$ is a Jordan automorphism.
\end{proof}

\begin{corollary}\label{o3}
Let $H$ be a  Hilbert space, 
let $1\le p\le\infty$, and 
let $S^p(H)$ be the algebra of $p$th Schatten class operators on $H$.
Then every continuous surjective local automorphism $T$ of $S^p(H)$ is a Jordan automorphism.
Further, if $H$ is an infinite-dimensional separable Hilbert space, then $T$ is an automorphism.
\end{corollary}
\begin{proof}
Our method consists in considering  the restriction of $T$ to the ideal $F(H)$ of $S^p(H)$.
Take $x\in F(H)$. Then there exists an automorphism $\Phi_x\colon S^p(H)\to S^p(H)$ such that
$T(x)=\Phi_x(x)$. By~\cite[Corollary~3.2]{Cher}, 
there exists a continuous invertible linear operator $u_x\colon H\to H$ such that
$\Phi_x(y)=u_xyu_x^{-1}$ for each $y\in S^p(H)$.
Consequently, $T(x)=u_x xu_x^{-1}\in F(H)$ and hence $T$ maps $F(H)$ into $F(H)$ and is a local automorphism of $F(H)$. 
Further, if $y\in F(H)$, then there exists an $x\in S^p(H)$ with $T(x)=y$. 
Since $T(x)=u_xxu_x^{-1}$, it follows that  $x=u_x^{-1}yu_x\in F(H)$.
This shows that $T$ is a surjective local automorphism of $F(H)$. Corollary~\ref{o2} tells us that $T$
is a Jordan automorphism of $F(H)$.
Since $F(H)$ is dense in $S^p(H)$ and $T$ is continuous on $S^p(H)$, it may be concluded that $T$ is a Jordan
automorphism of $S^p(H)$. 
From~\cite{Her} we see that $T$ is either an automorphism or an antiautomorphism.

We now suppose that $H$ is an infinite-dimensional separable Hilbert space,
let $(\xi_n)$ be an orthonormal basis of $H$, 
let $s\in S^p(H)$ be the weighted shift operator defined through
\[
s(\xi_n)=2^{-n}\xi_{n+1}\quad \]
for all $n\in\mathbb{N}$,
so that
\[
s^*(\xi_1)=0,
\quad
s^*(\xi_n)=2^{-n}\xi_{n-1}\quad n\ge 2,
\]
and let $u_s\colon H\to H$ be a continuous invertible linear operator such that
$T(s)=u_ssu_s^{-1}$.
Assume towards a contradiction that $T$ is an antiautomorphism. 
Take a conjugation $c$ on $H$ and define $\Phi$ on $S^p(H)$ by $\Phi(x)=T(cx^*c)$ for each $x\in S^p(H)$.
Then $\Phi$ is an automorphism of $S^p(H)$ and~\cite[Corollary~3.2]{Cher} shows that there exists
a continuous invertible linear operator $v\colon H\to H$ such that $\Phi(x)=vxv^{-1}$ for each $x\in S^p(H)$.
We thus get
\[
u_s s u_s^{-1}=T(s)=\Phi(cs^*c)=vcs^*cv^{-1}.
\]
Then the operator $u_s^{-1}vc$ is invertible, so that $\xi=u_s^{-1}vc\xi_1\ne 0$ and, on the other hand, 
\[
s\xi=u_s^{-1}vcs^*cv^{-1}u_s\xi=u_s^{-1}vcs^*\xi_1=0.
\]
This is a contradiction as $s$ is injective.
\end{proof}

\end{document}